\documentclass[12pt]{article}

\title{Another proof of Seymour's 6-flow theorem}

\author{
   Matt DeVos\thanks{
     Email: {\tt mdevos@sfu.ca}. 
     Supported by an NSERC Discovery Grant (Canada)}
 \and
   Jessica McDonald\thanks{
  Email: {\tt mcdonald@auburn.edu}.   
	Supported in part by  Simons Foundation Grant \#845698  
  }
 \and
   Kathryn Nurse\thanks{
  Email: {\tt knurse@sfu.ca} 
}
}

\date{}

\usepackage{fullpage}
\usepackage{graphicx}
\usepackage{amsfonts}
\usepackage{amsmath}
\usepackage{amsthm}
\usepackage{enumitem}

\theoremstyle{plain}
\newtheorem{theorem}{Theorem}

\theoremstyle{definition}

\newcommand\supp{\operatorname{supp}}

\begin{document}

\maketitle

\begin{abstract}
In 1981 Seymour proved his famous 6-flow theorem asserting that every 2-edge-connected graph has a nowhere-zero flow in the group ${\mathbb Z}_2 \times {\mathbb Z}_3$ (in fact, he offers two proofs of this result).  In this note we give a new short proof of a generalization of this theorem where ${\mathbb Z}_2 \times {\mathbb Z}_3$-valued functions are found subject to certain boundary constraints.
\end{abstract}

Throughout we permit loops and parallel edges.  Let $G = (V,E)$ be a digraph and let $v \in V$.  We define $\delta^+(v)$ $(\delta^-(v))$ to be the set of edges with tail (head) $v$.  Let $\Gamma$ be an abelian group written additively and let $\phi : E \rightarrow \Gamma$.  The \emph{boundary} of $\phi$ is the function $\partial \phi : V \rightarrow \Gamma$ given by the rule:
\[ \partial \phi(v) = \sum_{e \in \delta^+(v)} \phi(e) - \sum_{e \in \delta^-(v)} \phi(e). \]
Note that the condition $\sum_{v \in V} \partial \phi(v) = 0$ is always satisfied since every edge contributes 0 to this quantity.  
We say that $\phi$ is \emph{nowhere-zero} if $0 \not\in \phi(E)$ and we say that $\phi$ is a $\Gamma$-\emph{flow} if $\partial \phi$ is the constant 0 function.  Let us comment that reversing an edge and replacing the value assigned to this edge by its additive inverse preserves the boundary and maintains the condition nowhere-zero.  So, in particular, the question of when a graph has a nowhere-zero function $\phi : E \rightarrow \Gamma$ with a given boundary is independent of the orientation.  Setting ${\mathbb Z}_k = {\mathbb Z} / k {\mathbb Z}$ we may state Seymour's theorem as follows.

\begin{theorem}[Seymour \cite{PS}]
\label{6flow}
Every 2-edge-connected digraph has a nowhere-zero ${\mathbb Z}_6$-flow.  
\end{theorem}

This result combines with a theorem of Tutte \cite{WT} to show that every 2-edge-connected digraph has a nowhere-zero 6-flow (i.e. a ${\mathbb Z}$-flow with range a subset of $\{ \pm 1, \pm 2, \ldots, \pm 5\}$).  In this article we prove the following generalization of Seymour's theorem (set $T=U = \emptyset$ to derive Theorem \ref{6flow}).  Here $\mathrm{supp}(f)$ denotes the support of a function $f$ and for a graph $G$ and a set $X \subseteq V(G)$ we use $d(X)$ to denote the number of edges with exactly one end in $X$.

\begin{theorem}
\label{funny6}
Let $G = (V,E)$ be a connected digraph and let $T \subseteq U \subseteq V$ have $|T|$ even and $|U| \neq 1$.  Assume further that every  $\emptyset \neq V' \subset V$ with $V' \cap U = \emptyset$ satisfies $d(V') \ge 2$.  Then for $k=2,3$ there exist functions $\phi_k : E(G) \rightarrow  {\mathbb Z}_k$ satisfying the following properties:
\begin{itemize}
\item $(\phi_2(e), \phi_3(e)) \neq (0,0)$ for every $e \in E$,
\item $\mathrm{supp} \left( \partial \phi_2 \right) = T$, and
\item $\mathrm{supp} \left( \partial \phi_3 \right) = U$.
\end{itemize}
\end{theorem}

Under the stronger hypothesis that $G$ is 3-edge-connected, a theorem of Jaeger et. al.~\cite{JLPT} shows that one may find a nowhere-zero ${\mathbb Z}_6$-valued function with any desired zero-sum boundary function.  The main novelty of our result is that the hypotheses are relatively weak and the result has a quick proof by induction.  In particular, we do not require the standard reductions  to 3-connected cubic graphs.  

Our notation is fairly standard.  For sets $X,Y$ we use $X \oplus Y = (X \setminus Y) \cup (Y \setminus X)$ to denote symmetric difference.  If $G$ is a graph and $(G_1,G_2)$ is a pair of subgraphs satisfying $E(G_1) \cup E(G_2) = E(G)$, $E(G_1) \cap E(G_2) = \emptyset$, and $|V(G_1) \cap V(G_2)| = k$, then we call $(G_1, G_2)$ a $k$-separation.  This separation is \emph{proper} if $V(G_1) \setminus V(G_2) \neq \emptyset \neq V(G_2) \setminus V(G_1)$.  Note that a graph with at least $k+1$ vertices is $k$-connected if and only if it has no proper $(k-1)$-separation.  

\begin{proof}[Proof of Theorem \ref{funny6}] We proceed by induction on $|E|$.  The case when $|V| \le 2$ holds by inspection, so we may assume $|V| \ge 3$.  If there exists $v \in V \setminus U$ with $\mathrm{deg}(v) = 2$, then the result follows by contracting an edge incident with $v$ (to eliminate this vertex) and applying induction.  So we may assume no such vertex exists.  Next suppose that $G$ has a proper 1-separation $(G_1, G_2)$ with $\{v \} = V(G_1) \cap V(G_2)$.  First suppose that $U \subseteq V(G_1)$.  In this case the result follows by applying the theorem inductively to $G_1$ with the given sets $T,U$ and to $G_2$ with $T=U=\emptyset$.  Next suppose that $U$ contains a vertex in both $V(G_1) \setminus \{v\}$ and $V(G_2) \setminus \{v\}$.  For $i=1,2$ let $U_i = (V(G_i) \cap U) \cup \{v\}$ and choose $T_i = T \cap V(G_i)$ or $T_i = (T \cap V(G_i)) \oplus \{v\}$ so that $|T_i|$ is even.  For $i=1,2$ apply the theorem inductively to $G_i$ with $T_i, U_i$ to obtain the functions $\phi_2^i$ and $\phi_3^i$.  Now taking $\phi_2 = \phi_2^1 \cup \phi_2^2$ and a suitable choice of $\phi_3 = \phi_3^1 \cup \pm \phi_3^2$ gives the desired functions for $G$.  (To see this note that by choosing $\pm \phi_3^2$ we may arrange for $\partial \phi_3(v)$ to be zero or nonzero as desired).  

By the above arguments, we may now assume that $G$ is 2-connected.  If $U= \emptyset$, choose an edge $e=uw$ and apply the theorem inductively to $G' = G-e$ with $T' = \emptyset$ and $U' = \{u,w\}$ to obtain $\phi_2'$ and $\phi_3'$.  Extend $\phi_2'$ to a function $\phi_2:E \rightarrow {\mathbb Z}_2$ by setting $\phi_2(e) = 0$.  Since $\sum_{v \in V} \partial \phi_3'  = 0$ we have $\partial \phi_3'(u) = - \partial_3'(w) = \pm 1$.  Therefore, we may extend $\phi_3'$ to a function $\phi_3 : E \rightarrow {\mathbb Z}_3$ by setting $\phi_3(e) = \pm 1$ so that $\partial \phi_3 = 0$, thus completing the proof in this case.   

Now we may assume $|U| \ge 2$.  By Menger's theorem we may choose a nontrivial path $P$ so that both ends of $P$ are in $U$ and furthermore, some component,  say $H$, of $G-E(P)$  contains both ends of $P$.  Over all such paths, choose one $P$ so that $H$ is maximal.  Suppose (for a contradiction) that some component $H' \neq H$ of $G-E(P)$ satisfies $V(H') \cap U \neq \emptyset$.  Choose $u \in V(H') \cap U$ and choose two internally vertex-disjoint paths $Q_1, Q_2 \subseteq H'$ starting at $u$ and ending in $V(P)$.  Now we may choose a nontrival path $P' \subseteq P \cup Q_1$ with one end $u$ and the other an end of $P$ so that $H \cup Q_2$ is contained in some connected component of $G-E(P')$, thus contradicting our choice of $P$.  Therefore, every component of $G-E(P)$ apart from $H$ contains no vertices in $U$.  Note that by our choice, every interior vertex of $P$ is in $V \setminus U$ (and thus has degree $\ge 3$).

Let $\{u_1, u_2\}$ be the ends of $P$ and let $G' = G - E(P)$.  Define $T' = T \oplus \{u_1, u_2\}$ and $U' = U \cup V(P)$.  Now apply the theorem inductively to each component of $G'$ with the corresponding restrictions of $T'$ and $U'$ to obtain $\phi_2': E(G') \rightarrow {\mathbb Z}_2$ and $\phi_3':E(G') \rightarrow {\mathbb Z}_3$.  Extend $\phi_2'$ to a function $\phi_2 : E \rightarrow {\mathbb Z}_2$ by defining $\phi_2(e) = 1$ for every $e \in E(P)$ and note that $\supp( \partial \phi_2 ) = T$ as desired.  By possibly reorienting we may assume that $P$ is a directed path with edges in order $e_1, \ldots, e_k$.  By greedily assigning values to these edges in order, we may extend $\phi_3'$ to a function $\phi_3 : E \rightarrow {\mathbb Z}_3$ with the property that $\partial \phi_3(v) = 0$ for every internal vertex $v$ of $P$.  Now the function $\phi_3$ satisfies the desired boundary condition at every vertex except possibly the ends of $P$.  Let $t \in {\mathbb Z}_3$ and modify $\phi_3$ by adding $t$ to $\phi_3(e)$ for every $e \in E(P)$.  This has no effect on the boundaries of the internal vertices of $P$, and for some $t \in {\mathbb Z}_3$ the resulting function will have nonzero boundary at both ends of $P$.  This gives us our desired functions $\phi_2$ and $\phi_3$.
\end{proof}

\end{document}